\documentclass[journal,transmag]{IEEEtran}
\usepackage{amsmath}
\usepackage{amssymb}
\usepackage{graphicx}
\usepackage{color}

\newtheorem{definition}{Definition}
\newtheorem{lemma}{Lemma}
\newtheorem{remark}{Remark}
\newtheorem{theorem}{Theorem}

\newtheorem{proof}{Proof}
\newtheorem{example}{Example}

\begin{document}



\title{\LARGE \bf
Delay-robust control design for two heterodirectional linear coupled hyperbolic PDEs
}

\author{\IEEEauthorblockN{Jean Auriol\IEEEauthorrefmark{1}, Ulf Jakob F. Aarsnes\IEEEauthorrefmark{2}, Philippe Martin\IEEEauthorrefmark{1}
and Florent Di Meglio\IEEEauthorrefmark{1}}
\IEEEauthorblockA{\IEEEauthorrefmark{1} MINES ParisTech, PSL Research University, CAS - Centre automatique et syst\`emes, 60 bd St Michel 75006 Paris, France  \quad \quad \quad }%
\IEEEauthorblockA{\IEEEauthorrefmark{2}International Research Institute of Stavanger (IRIS), Oslo, Norway and DrillWell - Drilling and well centre for improved recovery\\
Stavanger, Norway}
\thanks{Corresponding author: J. Auriol (email: jean.auriol@mines-paristech.fr).}}
\IEEEtitleabstractindextext{
\begin{abstract}
We detail in this article the necessity of a change of paradigm for the delay-robust control of systems composed of two linear first order hyperbolic equations. One must go back to the classical trade-off between convergence rate and delay-robustness. More precisely, we prove that, for systems with strong reflections, canceling the reflection at the actuated boundary will yield zero delay-robustness. Indeed, for such systems, using a backstepping-controller, the corresponding target system should preserve a small amount of this reflection to ensure robustness to a small delay in the loop. This implies, in some cases, giving up finite time convergence.
\end{abstract}
}
\maketitle
\IEEEdisplaynontitleabstractindextext
\IEEEpeerreviewmaketitle
\section{Introduction}
\IEEEPARstart{I}n this paper, we highlight an important shortcoming of some control designs for systems of two heterodirectional linear first-order hyperbolic Partial Differential Equations (PDEs). More precisely, we show that imposing finite-time convergence by completely canceling the \textbf{proximal reflection} (i.e \textbf{the reflection at the actuated boundary}) yields, in some cases, zero robustness margins to arbitrarily small delays in the actuation path. In particular, the control laws in recent contributions (see for instance~\cite{auriol2016minimum,coron2017finite,coron2013local,di2013stabilization,hu2016control}) can have very poor to no robustness to delays due to the cancellation of the proximal reflection. To overcome this problem, 
we propose some changes in the design of target system to preserver a small amount of this reflection and ensure delay-robustness.

Most physical systems involving a transport phenomenon can be modeled using hyperbolic partial differential equations (PDEs): heat exchangers~\cite{xu2002exponential}, open channel flow~\cite{coron1999lyapunov},~\cite{de2003boundary}, multiphase flow~\cite{di2011dynamics},~\cite{dudret2012stability} or power systems~\cite{thorp1998electromechanical}. The backstepping approach~\cite{coron2013local,hu2016control} has enabled the design of stabilizing full-state feedback laws for these systems. The generalization of these stabilization results for a large number of systems has been a focus point in the recent literature (details in~\cite{auriol2016minimum,coron2017finite,coron2013local,hu2016control}). The main objective of these controllers is to ensure convergence in the minimum achievable time (as defined in~\cite{li2010strong}), thereby neglecting the robustness aspects that are essential for practical applications. Some of these questions have been the purpose of recent investigations: in the presence of uncertainties in the system, the design of adaptive control laws using filter or swapping design is the purpose of~\cite{anfinsen2016adaptive,anfinsen2016boundary}. A different approach, towards an engineering use of backstepping, consists in deriving sufficient conditions guaranteeing the exponential stability of the controlled system in presence of uncertainties~\cite{Auriol2017}.  
The issues of noise sensitivity and performance trade-off are considered in \cite{auriol2017performance} where a method enabling the tuning of observer and controller feedback aggressiveness is proposed. 
However, the impact on stability of small delays in the feedback loop, has not been studied yet in this context. It has been observed (see~\cite{datko1986example,logemann1996conditions}) that for many feedback systems, the introduction of arbitrarily small time delays in the loop may cause instability for any feedback. In particular, in~\cite{logemann1996conditions}, a systematic frequency domain treatment of this phenomenon for distributed parameter systems is presented. Here, we use these results to cast a new light on feedback control design for linear hyperbolic systems.

The main contribution of this article is a set of necessary and sufficient conditions for classical controllers for linear hyperbolic systems to be robust to small delays. We prove that finite-time stabilization by completely canceling the proximal reflection, often yields vanishing delay margins, making it an impractical control objective. Indeed, the controllers derived in~\cite{auriol2016minimum,coron2013local,hu2016control} can be unstable in presence of a small delay in the loop due to the cancellation of the proximal reflection. Furthermore, some systems (for which the product of the proximal and distal reflection gains is greater than one)
cannot be delay-robustly stabilized, irrespective of the method.
Specifically, we show that, for a system of two heterodirectional linear hyperbolic equations with anti-diagonal source terms\footnote{For systems with source terms on the diagonal, a transform is first employed that changes the reflection coefficients.}, \textit{if the product of the proximal and distal reflections is}:
\begin{itemize}
\item \textbf{Greater than one}, the system cannot be stabilized robustly to delays.
\item \textbf{Smaller than one but greater than one-half}, the system cannot be finite time stabilized robustly to delays.
\item \textbf{Smaller than one-half} the system can be finite-time stabilized robustly to delays.
\end{itemize}

Our approach is the following: considering the control law proposed in~\cite{coron2013local} and using a backstepping approach, the controlled system is mapped to a distributed delay equation. Then, using the Laplace transform we derive the closed-loop transfer function~\cite{curtain2009transfer}. It is shown to be potentially unstable in presence of small delays. 
To ensure delay-robust stabilization, we propose some adjustments in the control law proposed in~\cite{coron2013local} by means of an additional degree of freedom enabling a trade-off between convergence rate and delay robustness. More precisely, if the plant has some proximal reflection terms, the target system should preserve a small amount of this reflection.

An important by-product of this analysis, detailed in Section~\ref{transfo_back} is the reformulation of \emph{any} system of two coupled linear hyperbolic equations as a zero-order neutral system with distributed delay. This result is obtained via a backstepping change of coordinates and yields a new tool for the study of hyperbolic systems.

The paper is organized as follows. In Section~\ref{sec:tutorial} we illustrate the necessity of this change of paradigm with a well-known system of two transport equations. These results are extended in Section~\ref{sec:general} to coupled systems composed of two hyperbolic PDEs. A new control method is then derived in Section~\ref{sec:new}. The corresponding feedback system is proved to be stable to small delays. Finally some simulation results are given in Section~\ref{sec:simulation}.

\section{A tutorial example: transport equations}\label{sec:tutorial}
In this Section, we consider the tutorial example of two pure transport equations coupled at the boundaries. We first recall historical results on the delay-robust stabilizability of such systems~\cite{hale2002strong,logemann1996conditions}. Then, we study the behaviour of various control laws in the presence of small delays in the actuation path.

\subsection{Description of the system}
We consider the following linear hyperbolic system of two transport equations
\begin{align}
u_t(t,x)+\lambda u_x(t,x)&=0, \label{u_pde_no_cou}\\
v_t(t,x)-\mu v_x(t,x)&=0, \label{v_pde_no_cou}
\end{align}
evolving in~$\{(t,x) | \quad t>0,~ x \in [0,1] \}$, with the following linear boundary conditions
\begin{align}
u(t,0)=qv(t,0), \quad v(t,1)=\rho u(t,1)+U(t). \label{IC_no_cou}
\end{align}
We will use the term \textbf{proximal reflection} to denote $\rho$: the reflection at the actuated boundary, and \textbf{distal reflection} to denote $q$: the reflection on the unactuated boundary.
The boundary reflections ~$q \ne 0$ and~$\rho$, and the velocities~$\lambda$ and~$\mu$ are assumed to be constant. The control law is denoted~$U(t)$. Moreover, we assume that
\begin{align}
-\mu<0 <\lambda. 
\end{align}
The initial conditions denoted~$u_0$ and~$v_0$ are assumed to belong to~$L^2([0,1])$. In the following we define the \textbf{characteristic time} of the system~$\tau$ as
\begin{align}
\tau=\frac{1}{\lambda}+\frac{1}{\mu}. \label{eq_tau}
\end{align}
The product~$\rho q$, product of the proximal and distal reflections, is called \textbf{open-loop gain} of the system. We recall the following definition from~\cite{logemann1996conditions}.
\begin{definition}\label{robust_stabilisation} Delay-robust stabilization \cite{logemann1996conditions}. \\The controller~$U(t)=\mathcal{K}[\begin{pmatrix}u\\v\end{pmatrix}]$ where~$\mathcal{K} : (L^2)^2 \rightarrow \mathbb{R}$ is an operator, delay-robustly stabilizes the system~\eqref{u_pde_no_cou}-\eqref{IC_no_cou} in the sense of \cite{logemann1996conditions} if the resulting feedback system stabilizes the system~\eqref{u_pde_no_cou}-\eqref{IC_no_cou} in the sense of the $L^2$-norm and is delay-robustly stable with respect to small delays in the loop. A system is said to be delay-robustly stabilizable if and only if there exists such a~$\mathcal{K}$. 
\end{definition}
\subsection{Open-loop transfer function}
In this section, we consider~$v(t,1)$ as the output of the system \eqref{u_pde_no_cou}-\eqref{IC_no_cou}. Using the method of characteristics, one can easily prove that~$v(t,1)$ satisfies the following delay equation.
\begin{align}
v(t,1)=\rho q v(t-\tau,1)+U(t), \label{delay_eq_no_couplings}
\end{align}
where~$\tau$ is defined by \eqref{eq_tau}. 
In the following, we denote~$s$ the Laplace variable, and use boldface to denote the Laplace transform of a given function, i.e., the Laplace transform of~$v(t,x)$ will be denoted~$\boldsymbol{v}(s,x)$. Taking the Laplace transform of \eqref{delay_eq_no_couplings}, we get
\begin{align}
\boldsymbol{v}(s,1) = \frac{1}{1-\rho qe^{-\tau s}}\boldsymbol{U}(s) =: \boldsymbol{H}_0(s)\boldsymbol{U}(s).
\end{align}
The transfer function~$\boldsymbol{H}_0(s)$ is the open-loop transfer function of the system. Depending on the value of the open-loop gain~$\rho q$ this transfer function has either no pole in the Right Half Plane (RHP) or an infinite number of poles in the RHP. More precisely, if~$|\rho q|<1$, this transfer function has no pole in the RHP whereas if~$|\rho q| \geq 1$, it has an infinite number of poles \textbf{ whose real parts are positive}. They are defined as
\begin{align}
\left\{
    \begin{array}{ll}
        s = \frac{1}{\tau}\ln(\rho q)+\frac{2k\pi}{\tau} i \quad &\text{if~$\rho q>0$}, \\
				s=\frac{1}{\tau}\ln(|\rho q|)+\frac{(2k+1)\pi}{\tau} i\quad &\text{if~$\rho q<0$},
\end{array}
\right.
\end{align}
where~$k$ is an arbitrary integer. Consequently using~\cite[Theorem 1.2]{logemann1996conditions} we have the following theorem
\begin{theorem}
If~$|\rho q|\geq 1$, system \eqref{u_pde_no_cou}-\eqref{IC_no_cou} is not delay-robustly stabilizable.
\end{theorem}
This implies, in particular, that stability is equivalent to delay-robust stabilizability for system~\eqref{u_pde_no_cou}-\eqref{IC_no_cou}.
\subsection{Feedback control for an open-loop gain smaller than one}
\subsubsection{Finite-time stabilization}
In this section, we focus on a system of transport equations for which the open-loop gain satisfies~$|\rho q|<1$. Although in this case the system \eqref{u_pde_no_cou}-\eqref{IC_no_cou} is already exponentially stable in open-loop, one could want to increase the convergence rate or to have finite-time convergence. This improvement of the controller performance can be done using impedance matching methods (see~\cite{aarsnes2013limits,aarsnes2014modeling,egeland2002modeling}). This method is used, for instance, to improve the control performance for the heave rejection problem (\cite{aarsnes2014modeling}), one can match the load impedance (the pressure to flow ration in the frequency domain at the boundary) to the characteristic line impedance (the pressure to flow ratio in the frequency domain in the transmission line). The application of this method in the case of transport equations consists in canceling totally the reflexion term~$\rho u(t,1)$, and get a semi-infinite system that converges to zero in finite time. The corresponding control law is then defined by
\begin{align}
U(t)=-\rho u(t,1). \label{Control_law_transport}
\end{align}
Consider now that there is a small delay~$\delta>0$ in the actuation. The output~$v(t,1)$ is then solution of the following delay equation
\begin{align}
v(t,1)&=\rho q v(t-\tau,1)+ U(t-\delta), \nonumber\\
&=\rho q v(t-\tau,1)- \rho q v(t-\tau-\delta,1),  \label{delay_eq_no_couplings2}
\end{align}
  which is a zero-order scalar neutral system. Using classical results on such systems~\cite{hale2002strong}, we get that, a necessary condition to have equation \eqref{delay_eq_no_couplings2} stable for any delay~$\delta>0$ is 
\begin{align}
|\rho q|+|\rho q|<1 \Leftrightarrow |\rho q| <\frac{1}{2}. \label{eq_rho_q}
\end{align}
This means that for an open-loop gain such that~$|\rho q| \geq  \frac{1}{2}$, one cannot have both robustness to a delay and \textbf{finite time} convergence. This justifies the observations that have been done by industrial practitioners about the limitations of the impedance matching method. For instance, in~\cite{kyllingstad2009new}, the authors design a controller preventing stick-slip oscillations of a drill-string (a dysfunction of rotary drilling, characterized by large cyclic variations of the drive torque and the rotational bit speed). They observed that completely canceling the proximal reflection coefficient could change the dynamics of the string in a way that makes the system unstable. This observation is confirmed by the present analysis.

\subsubsection{First solution: preserving some reflection}\label{SubTransEq}
Although canceling the proximal reflection to stabilize system \eqref{u_pde_no_cou}-\eqref{IC_no_cou} increases the nominal convergence rate, the corresponding feedback system is not robustly stable to delay in the loop when~$|\rho q| \geq \frac{1}{2}$. Thus, it appears necessary to keep some reflection terms. To do so, let us slightly change the control law and replace \eqref{Control_law_transport} by
\begin{align}
U(t)=-Ku(t,1), \quad \quad  \label{U_transport}
\end{align}
In presence of a delay~$\delta>0$ in the actuation, we get, for the closed-loop system, the following delay equation:
\begin{align}
v(t,1)-\rho q v(t-\tau ,1)+K q v(t-\tau-\delta,1)=0, \label{delay_transport2}
\end{align}
which is now exponentially stable (\cite{hale2002strong}) for all~$\delta>0$ if~$K$ satisfies the following equation:
\begin{align}
 |K| < \frac{1-|\rho q|}{|q|}. 
\end{align}
 Note that such a $K$ exists since $ \frac{1-|\rho q|}{|q|}>0$  (as we assumed $|\rho q| <1$).
\subsubsection{Second solution: filtering the control law} \label{filter}
A second approach to provide a delay-robust stabilization consists in filtering the control law. Let us consider e.g the control law~$U_1(t)$ defined by its Laplace transformation as
\begin{align}
\boldsymbol{U}_1(s)=-\frac{1+as}{1+bs}\rho\boldsymbol{u}(1,s).
\end{align} 
where~$a$ and~$b>0$ are some coefficients that have still to be defined. In presence of a delay~$\delta>0$ in the actuation, we get, for the closed-loop system, the following delay equation:
\begin{align}
&\dot{v}(t,1)-\rho q \dot{v}(t-\tau ,1)+ \frac{a}{b}\rho q \dot{v}(t-\tau-\delta,1)=\nonumber\\
&-\frac{1}{b}(v(t,1)-\rho q v(t-\tau ,1)+ q\rho a v(t-\tau-\delta,1)), \label{delay_transport3}
\end{align}
which is exponentially stable (see \cite{fabiano2013stability}) for all~$\delta>0$ if
\begin{align}
\frac{a}{b}<\frac{1-|\rho q|}{|\rho q|},\quad a<\frac{1-|\rho q|}{|\rho q|}. \label{cond_a_b}
\end{align} 
\subsubsection{Concluding remarks}
Throughout the analysis of a simple system of two transport equations, we have proved (using the results from~\cite{logemann1996conditions}) that there is a whole class of systems (the ones for which the product~$|\rho q|$ is larger than one) that cannot be robustly stabilized in presence of a small delay in the loop. Moreover, it has appear that even for~$|\rho q|<1$, finite-time convergence is not a reasonable  objective since the corresponding controller is not always robust to delays.
This means that to have delay-robustness one may have to give up finite-time convergence.
In the next section, we show that this change of paradigm still holds for general coupled system of two hyperbolic partial differential equations.
\section{General case of two coupled equations}\label{sec:general}
In this section we consider the general case of two linear coupled hyperbolic PDEs. The main objective is to prove that the results of the previous section (requiring giving up finite time convergence to obtain delay-robust stabilization) still holds in the case where in-domain couplings exist. To do so, using a classical backstepping transformation (see~\cite{coron2013local}), the original system is mapped to a distributed-delay neutral system. Deriving the corresponding transfer function, it becomes possible to adjust the results of the previous section to this general case. Interestingly, this approach highlights the potential of backstepping as an analysis tool, rather than just a control design tool.
\subsection{Description of the system}
We consider the following linear hyperbolic system which appear in Saint-Venant equations, heat exchangers equations and other linear hyperbolic balance laws (see~\cite{bastin2016stability}).
\begin{align}
u_t(t,x)+\lambda u_x(t,x)&=\sigma^{+-}(x)v(t,x), \label{u_pde}\\
v_t(t,x)-\mu v_x(t,x)&=\sigma^{-+}(x)u(t,x), \label{v_pde}
\end{align}
evolving in~$\{(t,x) | \quad t>0,~ x \in [0,1] \}$, with the following linear boundary conditions 
\begin{align}
u(t,0)=qv(t,0), \quad v(t,1)=\rho u(t,1)+U(t), \label{IC}
\end{align}
The inside-domain coupling terms~$\sigma^{-+}$ and~$\sigma^{+-}$ can be spatially-varying, whereas the boundary coupling terms~$q \ne 0$ and~$\rho$, and the velocities~$\lambda$ and~$\mu$ are assumed to be constant. Moreover, we assume that
\begin{align}
-\mu<0 <\lambda. 
\end{align}
The initial conditions denoted~$u_0$ and~$v_0$ are assumed to belong to~$L^2([0,1])$. This system is pictured in Figure~\ref{Modele}.
\begin{figure}%
	\includegraphics[width=\columnwidth,trim= 0 4cm 0 4cm]{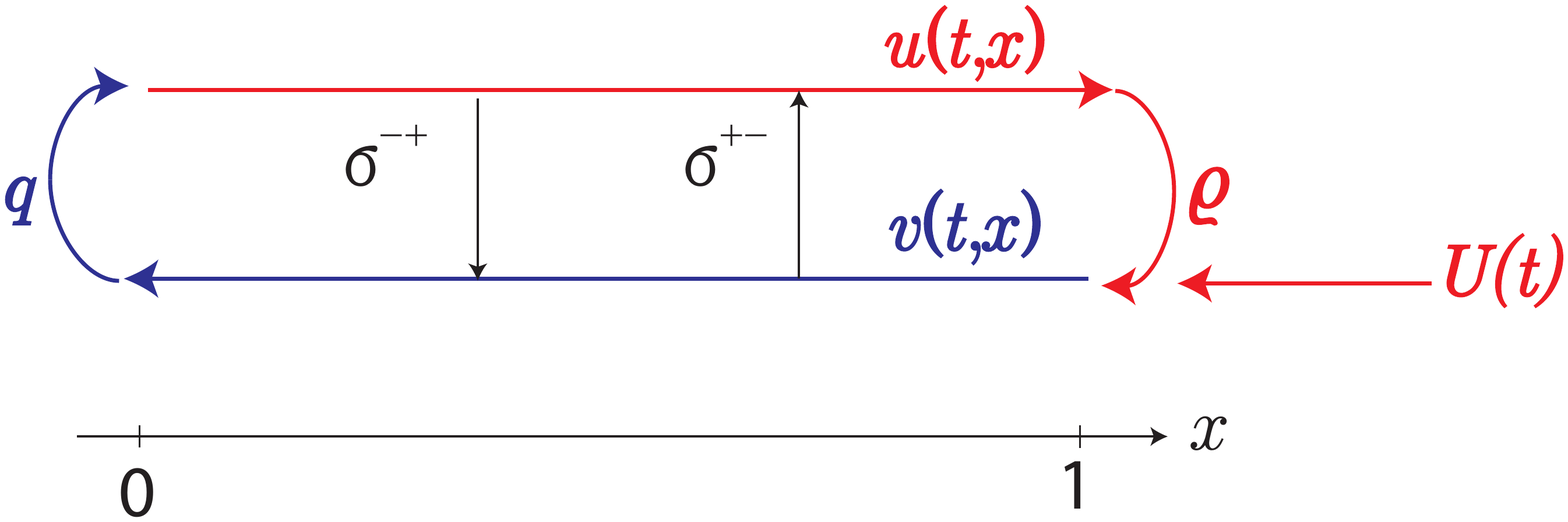}%
	\caption{Schematic representation of system \eqref{u_pde}-\eqref{IC}}%
	\label{Modele}%
\end{figure}
\subsection{A distributed-delay differential equation}\label{transfo_back}
In this section, by means of a classical backstepping transformation, the original system \eqref{u_pde}-\eqref{IC} is mapped to a neutral system with distributed-delay.
\subsubsection{Volterra transformation: removing inside-domain couplings} 
 We consider the following Volterra change of coordinates defined in~\cite{coron2013local} by
\begin{align}
&\alpha(t,x)=u(t,x) \nonumber\\
&\quad \quad -\int^x_0(K^{uu}(x,\xi)u(\xi)+K^{uv}(x,\xi)v(\xi))d\xi,\label{back1} \\
&\beta(t,x)=v(t,x) \nonumber \\
&\quad \quad -\int^x_0(K^{vu}(x,\xi)u(\xi)+K^{vv}(x,\xi)v(\xi))d\xi, \label{back2}
\end{align}
where the kernels~$K^{uu}, K^{uv}, K^{vu}, K^{vv}$ are defined on ~$\mathcal{T}=\{(x,\xi)\in [0,1]^2 |\quad \xi \leq x\}$ by the following set of hyperbolic PDEs:
\begin{align}
\lambda K^{uu}_x(x,\xi)+\lambda K^{uu}_{\xi}(x,\xi) &=-K^{uv}(x,\xi)\sigma^{-+}(\xi), \label{kerneK^{uu}}\\
\lambda K^{uv}_x(x,\xi)-\mu K^{uv}_{\xi}(x,\xi) &=-K^{uu}(x,\xi)\sigma^{+-}(\xi), \label{kerneK^{uv}}\\
\mu K^{vu}_x(x,\xi)-\lambda K^{vu}_{\xi}(x,\xi)&=K^{vv}(x,\xi)\sigma^{-+}(\xi), \label{kerneK^{vu}}\\
\mu K^{vv}_x(x,\xi)+\mu K^{vv}_{\xi}(x,\xi)&= K^{vu}(x,\xi)\sigma^{+-}(\xi), \label{kerneK^{vv}}
\end{align}
with the following set of boundary conditions:
\begin{align}
K^{vu}(x,x)=-\frac{\sigma^{-+}(x)}{\lambda+\mu},~ K^{vv}(x,0)=\frac{\lambda q}{\mu}K^{vu}(x,0),\\
K^{uv}(x,x)=\frac{\sigma^{+-}(x)}{\lambda+\mu},~ K^{uv}(x,0)=\frac{\lambda q}{\mu}K^{uu}(x,0). \label{K_bond}
\end{align}
\begin{lemma}[~\cite{coron2013local}]
Consider system \eqref{kerneK^{uu}}-\eqref{K_bond}. There exists a unique solution~$K^{uu}$,~$K^{uv}$,~$K^{vu}$ and~$K^{vv}$ in~$L^{\infty}(\mathcal{T})$. Moreover, transformation~\eqref{back1}-\eqref{back2} is invertible and this inverse transformation can be expressed as follow
\begin{align}
&u(t,x)=\alpha(t,x) \nonumber\\
&\quad \quad +\int^x_0(L^{\alpha\alpha}(x,\xi)\alpha(\xi)+L^{\alpha\beta}(x,\xi)\beta(\xi))d\xi\label{back_inv1} \\
&v(t,x)=\beta(t,x) \nonumber \\
&\quad \quad +\int^x_0(L^{\beta\alpha}(x,\xi)\alpha(\xi)+L^{\beta\beta}(x,\xi)\beta(\xi))d\xi \label{back_inv2}
\end{align}
where the kernels $L^{\alpha\alpha}$,~$L^{\alpha\beta}$,~$L^{\beta\alpha}$ and~$L^{\beta\beta}$ belongs to~$L^{\infty}(\mathcal{T})$ 
\end{lemma}
The dynamics of the system in the new coordinates is:
\begin{align}
\alpha_t(t,x)+\lambda \alpha_x(t,x)&=0, \label{alpha_pde}\\
\beta_t(t,x)-\mu \beta_x(t,x)&=0, \label{beta_pde}
\end{align}
with the following linear boundary conditions 
\begin{align}
\alpha(t,0)&=q\beta(t,0), \nonumber \\
 \beta(t,1)&=\rho \alpha(t,1)+U_0(t)\nonumber\\
&-\int_0^1N^{\alpha}(\xi)\alpha(t,\xi)+N^{\beta}(\xi)\beta(t,\xi)d\xi, \label{IC_alpha}
\end{align}
with
\begin{align}
&N^{\alpha}(\xi)=L^{\beta\alpha}(1,\xi)-\rho L^{\alpha\alpha}(1,\xi), \label{L_alpha} \\
&N^{\beta}(\xi)=L^{\beta\beta}(1,\xi)-\rho L^{\alpha\beta}(1,\xi). \label{L_beta}
\end{align}

\subsubsection{Neutral equation with distributed delay}
Using the method of characteristics on equations \eqref{alpha_pde}-\eqref{beta_pde} yields (for all~$x \in [0,1]$, for all~$t>0$)
\begin{align}
\alpha(t,x)&=q\beta(t-\frac{x}{\lambda}-\frac{1}{\mu},1), \label{eq_alpha_1}\textbf{}\\
\beta(t,x)&=\beta(t-\frac{1-x}{\mu},1). \label{eq_beta_1}
\end{align}
Consequently, combining equations~\eqref{eq_alpha_1}-\eqref{eq_beta_1} and equation~\eqref{IC_alpha}, we get:
\begin{align}
\beta(t,1)&=q\rho \beta(t-\tau,1)+U_0(t)\nonumber\\
&-\int_0^1N^{\alpha}(\xi) \beta(t-\frac{\xi}{\lambda}-\frac{1}{\mu},1)d\xi \nonumber \\
&-\int_0^1 N^{\beta}(\xi) \beta(t-\frac{1-\xi}{\mu},1) d\xi \nonumber 
\end{align}
Thus,
\begin{align}
\beta(t,1)&=q\rho \beta(t-\tau,1)+U_0(t) \nonumber\\
&-\int_0^\tau \tilde{N}(\nu) \beta(t-\nu,1) d\nu,
\label{delay_eq_alpha}
\end{align}
where~$\tau$ is defined by \eqref{eq_tau} and where~$\tilde{N}$ is defined by
\begin{align}
\tilde{N}(\nu)=\lambda N^\alpha(\lambda\nu-\frac{\lambda}{\mu})h_{[\frac{1}{\mu},\tau]}(\nu)+\mu N^\beta(1-\mu \nu)h_{[0,\frac{1}{\mu}]}(\nu) \label{eq_N_tilde}
\end{align}
where for any interval I,~$h_I(x)$ is defined by
\begin{align}
h_I(x)=\left\{
    \begin{array}{ll}
        1 & \quad \mbox{if,} \quad x \in I \\
        0 & \quad \mbox{else.}
            \end{array}
\right.
\end{align}
This invertible coordinate change enable us to rewrite $\beta$ as the solution of a delay equation with distributed delays. Since this transformation is independent of the control law, it means that the class of systems described by \eqref{u_pde}-\eqref{IC}, is equivalent to a class of neutral systems with distributed delay, as given by \eqref{delay_eq_alpha}.

\begin{remark}\em
This result is crucial in so far as it offers a new outlook to analyze the properties of hyperbolic systems. Existing stability results for neutral equations (see \cite{damak2015stability,hale2002strong}) can for instance be adjusted for hyperbolic PDEs, due to equation~\ref{delay_eq_alpha}. 
\end{remark}
\begin{remark}\em
 The equivalence between systems described by a single first-order hyperbolic PDE and systems described by integral delay equations was already proved in \cite{karafyllis2014relarion}. Equation~\eqref{delay_eq_alpha} extends this result for system composed of two coupled hyperbolic PDEs.
\end{remark}


\subsection{Open-loop analysis}
In this section, we consider~$\beta(t,1)$ as the output of the system \eqref{alpha_pde}-\eqref{beta_pde}. Taking the Laplace transform of \eqref{delay_eq_alpha}, we get
\begin{align}
\boldsymbol{\beta}(s,1)&=\frac{1}{1-\rho q e^{-\tau s}+\int_0^\tau \tilde{N}(\xi)e^{-\xi s}d\xi}\boldsymbol{U}_0(s) \nonumber\\
 & :=\boldsymbol{H}_1(s)\boldsymbol{U}_0(s). 
\end{align}

We then have the following lemma
\begin{lemma}\label{Lem_inf_poles}
If~$|\rho q|> 1$, then the open-loop transfer function~$\boldsymbol{H}_1(s)$ has an infinite number of poles with a positive real part.
\end{lemma}
\begin{proof}\em
We assume here that the open loop gain~$\rho q$ is positive (the case~$\rho q<0$ can be treated in a similar way).
The poles of the open-loop transfer function are the solutions of 
\begin{align}
1-\rho q e^{-\tau s}+\int_0^\tau \tilde{N}(\xi)e^{-\xi s}d\xi=0,
\end{align}
$\tilde{N}$ is differentiable almost everywhere on~$[0,1]$. Integrating by parts yields
\begin{align}
0&=1-\rho qe^{-s\tau}+\frac{1}{s}\int_0^\tau \tilde{N}'(\xi)e^{-\xi s}d\xi -\frac{1}{s}J(s),
\end{align}
where~$J$ is a bounded function that represents the jumps appearing in the integration by parts.
We denote in the following
\begin{align}
F(s)&=1-\rho qe^{-s\tau},\\
H(s)&=\frac{1}{s}\int_0^\tau \tilde{N}'(\xi)e^{-\xi s}d\xi -\frac{1}{s}J(s).
\end{align}
Since~$\tilde{N}'$ is bounded, the function~$|H|$ converges to 0 for~$|s|\rightarrow 0$ . The function~$F$ has an infinite number of zeros whose real parts are equal to~$\frac{\ln (\rho q)}{2\tau}$. The hypothesis of Theorem~\ref{Rouche_general} (see the Appendix) are satisfied and we can then conclude that~$F+H$ has an infinite number of zeros whose real parts are strictly positive. This concludes the proof
\end{proof}
We can now state the following Theorem
\begin{theorem}\label{TH2}
If~$|\rho q|> 1$, system \eqref{u_pde}-\eqref{IC} cannot be stabilized robustly to delays.
\end{theorem}
\begin{proof}\em
If there exists a controller~$U_0$ for system \eqref{u_pde}-\eqref{IC} such that the resulting feedback system is robustly stable to small delays in the loop, it implies that equation \eqref{delay_eq_alpha} is stable (since both system are equivalent). It means that there exists a controller for system \eqref{delay_eq_alpha} such that the resulting feedback system is robustly stable to small delays in the loop. This is impossible (see \cite[Theorem 1.2]{logemann1996conditions}) since the open-loop transfer function has an infinite number of poles with a positive real part.
\end{proof}
We have proved in this section that, similarly to transport equations, if the open-loop gain~$|\rho q|$ is greater than one, one cannot find a controller whose delay margin is non-null. Consequently, there is a whole class of hyperbolic systems that cannot be delay-robustly stabilized.

\begin{remark}\em
The critical case~$\rho q=1$ is not considered here. Indeed one cannot simply adjust the previous proof, since the zeros of~$F$ are located on the imaginary axis.
\end{remark}
\subsection{Feedback control for an open loop gain smaller than one}

\subsubsection{Finite-time stabilization}

In this section, we focus on a system of hyperbolic equations for which the open-loop gain satisfies~$|\rho q|<1$. Note that  the uncontrolled system can be unstable due to the inside-domain couplings~$\sigma^{-+}$ and~$\sigma^{+-}$ (see~\cite{bastin2016stability}). In~\cite{coron2013local}, using the backstepping approach, a control law that ensures finite-time stabilization of the original system was derived. This control law is defined by
\begin{align}
	&U_{BS}(t)=-\rho u(t,1)\nonumber \\
	&+\int_0^1K^{vu}(1,\xi)u(\xi,t)+K^{vv}(1,\xi)v(\xi,t)d\xi \nonumber \\
	&=-\rho \alpha (t,1) \nonumber \\
	&+\int_0^1N^{\alpha}(1,\xi)\alpha(\xi,t)+N^{\beta}(1,\xi)\beta(\xi,t)d\xi, \label{control}
	\end{align}
where the kernels~$K^{vu}$ and~$K^{vv}$ are defined by equations \eqref{kerneK^{uu}}-\eqref{K_bond} and~$N^{\alpha}, N^\beta$ are defined by equations \eqref{L_alpha}-\eqref{L_beta}.

Consider now that there is a small delay~$\delta>0$ in the actuation. The output~$\beta(t,1)$ is then solution of the following delay equation
\begin{align}
&\beta(t,1)=q\rho \beta(t-\tau,1)-\rho q \beta(t-\tau-\delta,1)\nonumber\\
&-\int_0^\tau \tilde{N}(\xi) (\beta(t-\xi,1) - \beta(t-\delta-\xi,1)) d\xi ,
\label{delay_eq_alpha_bis}
\end{align}
where~$\tilde{N}$ is defined by \eqref{eq_N_tilde}.
We denote~$\boldsymbol{H}_1(s)$ as
\begin{align}
\boldsymbol{H}_1(s)=\int_0^1 \tilde{N}(\xi)(e^{-\xi s}-e^{-(\xi+\delta)s})d\xi.
\end{align}
Taking the Laplace transform of equation \eqref{delay_eq_alpha_bis}, we get
\begin{align}
(1-\rho q e^{-\tau s} +\rho q e^{-(\tau+\delta) s})\boldsymbol{\beta}(s,1)=\boldsymbol{H}_1(s)\boldsymbol{\beta}(s,1).
\end{align}
We can now state the following Theorem:
\begin{theorem} \label{TH3} If~$|\rho q| > \frac{1}{2}$, then the system \eqref{u_pde}-\eqref{IC} with the delayed backstepping control law~$U_{BS}(t-\delta)$ is unstable for any~$\delta>0$.
\end{theorem}
\begin{proof}\em
This proof uses the same idea as the one used for the proof of Lemma~\ref{Lem_inf_poles}. Let us denote 
\begin{align}
F_1(s)=1-\rho q e^{-\tau s}+\rho q e^{-(\tau +\delta)s},
\end{align}
and 
\begin{align}
F_2(s)=1-\rho q e^{\tau \epsilon} e^{-\tau s}+\rho q e^{(\tau +\delta)\epsilon} e^{-(\tau +\delta)s},
\end{align}
where~$\epsilon>0$. Choosing~$\epsilon$ small enough, we have that~$|\rho q e^{\tau \epsilon}|+|\rho q e^{(\tau +\delta)\epsilon}|>1$. Consequently,~$F_2(s)$ has an infinite number of roots whose real parts are positive (see~\cite{hale2002strong}). Moreover, these roots are unbounded. Thus,~$F_1(s)$ has an infinite number of roots whose real parts are larger than~$\epsilon$. Integrating by part~$H_1(s)$ we prove that~$|H_1(s)|$ converges to zero for~$|s|$ large enough. Using Theorem~\ref{Rouche_general}, we have that~$F_1+H_1$ has at least one root whose real part is strictly positive. This concludes the proof.
\end{proof}
The fact that the backstepping controller proposed in~\cite{coron2013local} has \textit{zero} delay margin when~$|\rho q| > \frac{1}{2}$ means that it cannot be used for practical applications.
Specifically, $|\rho q| > \frac{1}{2}$ indicates that the feedback systems cannot have both finite time convergence and be robust to delays. Similarly to the case of transport equations, this stability limitation is not due to the backstepping method itself bu is strongly interwoven with the cancellation of the proximal reflection term~$\rho u(t,1)$. To obtain a tractable implementation of a controller for the system \eqref{u_pde}-\eqref{IC}, one must have robustness to delays and thereby give up finite-time convergence.

\begin{remark}\em
For systems that do not have a reflection at either boundary, there is no concern with delay-robustness. This is consistent with the delay-robustness results for predictor feedback developed in \cite{bresch2010delay,krstic2008lyapunov}.
\end{remark}
In the next section we propose a different control design by slightly adjusting the control law \eqref{control}.

\section{A new control paradigm}\label{sec:new}

In this section we slightly modify the control law \eqref{control} to overcome the stability limitation exposed above, while maintaining the same structure for the controller. The control law \eqref{control} is composed of two parts:
\begin{enumerate}
	\item the integral part whose objective is to remove the effect of inside-domain couplings\\
	\item the term~$-\rho u(t,1)$ whose objective is to cancel the proximal reflection and to ensure finite-time convergence.
\end{enumerate}
As seen above, the instability of the feedback system in presence of small delay in the loop is mostly due to the term~$-\rho u(t,1)$ in the control law. It appears consequently necessary to avoid the total cancellation of the proximal  reflection (and thereby giving up finite time convergence). Based on the analysis carried out in Section~\ref{SubTransEq} for the case of transport equations, we proposed a similar adjustment for the control law \eqref{control} when~$|\rho q | <1$.


\subsection{Control law}
Let us consider the following control law:
\begin{align}
	&U_{BS_2}(t)=-K u(t,1)\nonumber \\
	&-(\rho-K)\int_0^1K^{uu}(1,\xi)u(\xi,t)+K^{uv}(1,\xi)v(\xi,t)d\xi \nonumber \\
	&+\int_0^1K^{vu}(1,\xi)u(\xi,t)+K^{vv}(1,\xi)v(\xi,t)d\xi, \label{control3}
\end{align}
where~$K^{uu}$,~$K^{uv}$,~$K^{vu}$ and~$K^{vv}$ are defined by equations \eqref{kerneK^{uu}}-\eqref{K_bond} and where, similarly to \eqref{U_transport}, the coefficient~$K$ is chosen such that
\begin{align}
|K|< \frac {1-|\rho q|}{|q|}. \label{boun_K2}
	\end{align}
	The objective of such a control law is preserve a small amount of proximal reflection in the target system to ensure delay-robustness, while eliminating inside-domain couplings.
	\begin{remark}\em
	The control law~$U_{BS_2}$ can be rewritten as 
	\begin{align}
	U_{BS_2}&=-K \alpha(t,1) \nonumber \\
	&+\int_0^1 N^{\alpha}(1,\xi) \alpha(t,\xi)+N^\beta(1,\xi) \beta(t,\xi) d\xi,
	\end{align}
	\end{remark}
	Using the backstepping transformation \eqref{back1}-\eqref{back2}, the system \eqref{u_pde}-\eqref{IC} is mapped to 
\begin{align}
\alpha_t(t,x)+\lambda \alpha_x(t,x)&=0, \label{alpha_eq_mille}\\
\beta_t(t,x)-\mu \beta_x(t,x)&=0,
\end{align}
with the boundary conditions
\begin{align}
&\alpha(t,0)=q\beta(t,0), \nonumber \\
 &\beta(t,1)=(\rho-K) \alpha(t,1). \label{alpha_boun_mille}
\end{align}

\begin{lemma}
The system \eqref{alpha_eq_mille}-\eqref{alpha_boun_mille} is exponentially stable
\end{lemma}
\begin{proof}\em
It is sufficient to prove that~$|q(\rho-K)|<1$. To do so, let us consider all the cases depending on the signs of~$q$ and~$\rho$. If~$\rho>0$ and~$q>0$, we have (using \eqref{boun_K2})
\begin{align}
-1+2\rho q<(\rho-K)q<1 \Rightarrow |(\rho-K)q|<1,
\end{align}
since~$|\rho q|<1$. The other cases can be treated similarly.
\end{proof}
Consequently, the proposed control law stabilizes exponentially the system \eqref{u_pde}-\eqref{IC}. 

\begin{remark}\em
The coefficient~$K$ can be interpreted as a tuning parameter, enabling a trade-off between performance (convergence rate) and robustness with respect to delays. This parameter has a role similar to the coefficient~$\epsilon$ introduced in \cite{auriol2017performance} in the design of the observer to enable a trade-off between performance and noise sensitivity.
\end{remark}

\begin{remark}\em
Another approach to delay-robustly stabilize \eqref{u_pde}-\eqref{IC} would consist in filtering the control law \eqref{control} (as proposed in section~\ref{filter}). This is discussed in Section~\ref{Complements}.
\end{remark}
We need now to prove that the proposed control law is robust with respect to small delays. We have the following theorem.
\begin{theorem}
Consider the control law~$U_{BS_2}$ defined by \eqref{control3} with~$K$ satisfying \eqref{boun_K2}. This control law delay-robustly stabilizes the system \eqref{u_pde}-\eqref{IC} in the sense of Definition~\ref{robust_stabilisation}.
\end{theorem}
\begin{proof}\em
Consider a positive delay~$\delta$. Consider the two states~$\alpha$ and~$\beta$ defined by \eqref{back1}-\eqref{back2}. Slightly adjusting the method used to derive \eqref{delay_eq_alpha_bis}, we get the following equation satisfied by the output~$\beta(t,1)$.
\begin{align}
&\beta(t,1)=q\rho \beta(t-\tau,1)-qK \beta(t-\tau-\delta,1)\nonumber\\
&-\int_0^\tau \tilde{N}(\xi) (\beta(t-\xi,1) - \beta(t-\delta-\xi,1)) d\xi, 
\label{delay_eq_alpha_ter}
\end{align}
where~$\tilde{N}$ is defined by \eqref{eq_N_tilde}. Taking the Laplace transform yields the following characteristic equation
\begin{align}
\boldsymbol{F}(s)=1-q\rho e^{-\tau s} +K q e^{-(\tau+\delta) s} -I(s,\delta)=0,
\end{align}
where~$I(s, \delta)$ is defined by 
\begin{align}
I(s,\delta)=\int_0^\tau \tilde{N}(\xi) (e^{-\xi s}-e^{-(\xi+\delta)s})  d\xi.
\end{align}
Let us now consider a complex number~$s$ such that~$\Re(s) \geq 0$. We then have
\begin{align}
|\boldsymbol{F}(s)| &\geq  |1-q\rho e^{-\tau s} +K q e^{-(\tau+\delta) s}| -|I(s,\delta)| \nonumber \\
&\geq  1- |q\rho e^{-\tau s}| -|K q e^{-(\tau+\delta) s}| -|I(s,\delta)|\nonumber\\
&\geq 1- |q\rho| -|K q|  -|I(s,\delta)|.
\end{align}
Since~$K$ satisfies \eqref{boun_K2}, there exists~$\epsilon_0>0$ such that 
\begin{align}
1- |q\rho| -|K q|>\epsilon_0.
\end{align}
\color{red}
Let us now focus on the term $I(s,\delta)$. We have 
\begin{align}
|I(s,\delta)|^2=(1+e^{-2\delta x}-2e^{-\delta x}\cos(\delta y))\cdot|\int_0^\tau N(\xi)e^{-\xi(x+iy)}d\xi|^2. \nonumber
\end{align}
The integral on the expression above can be rewritten
\begin{align}
&\int_0^\tau N(\xi)e^{-\xi(x+iy)}d\xi = \mathcal{F}(N(\cdot)e^{ -\cdot x})(y),
\end{align}
where we denote $\mathcal{F}$ the Fourier transform. Since $N(\cdot)$ belong to $L^1(0,\tau)$, this yields 
\begin{align}
\forall x>0,~ \exists M_0 \in \mathbb{R},~ \forall |y| > M_0,~ |\int_0^\tau N(\xi)e^{-\xi(x+iy)}d\xi| < \frac{\epsilon_0}{2} \nonumber
\end{align}
Moreover, one can easily prove (Lebesgues) that 
\begin{align}
\exists M_1 \in \mathbb{R^+},~ \forall |x| > M_1,~\forall y \in \mathbb{C},~ |\int_0^\tau N(\xi)e^{-\xi(x+iy)}d\xi| < \frac{\epsilon_0}{2} \nonumber
\end{align}
We can now choose~$\delta_0$ small enough such that for any~$\delta \leq \delta_0$, $\forall x \in [0,M_1],~\forall y \in [-M_0, M_0],~|I(s,\delta)|<\epsilon_0$. With this choice of $\delta_0$, one can easily check that, $\forall \delta \leq \delta_0$, $\forall s \in \mathbb{C}$ such that~$\Re(s) \geq 0$
\begin{align}
|I(s,\delta)|< \epsilon_0
\end{align}
\color{black}
Consequently, for~$\delta \leq \delta_0$, we have
\begin{align}
|\boldsymbol{F}(s)|>0.
\end{align}
It means that  for~$0<\delta \leq \delta_0$, the function~$\boldsymbol{F}(s)$ does not have any root whose real part is positive. Consequently, equation \eqref{delay_eq_alpha_ter} is asymptotically stable. Thus, using the invertibility of the Volterra transformation \eqref{back1}-\eqref{back2}, this concludes the proof.
\end{proof}
\begin{remark}\em
For a given value of~$K$, the parameter~$\delta_0$ gives a range for admissible delays. However,~$\delta_0$ is not necessarily the maximum admissible delay.
\end{remark}
\subsection{Interpretation of the results and outlook}\label{Complements}
In this section, we analyze the practical consequences of Theorems~\ref{TH2} and~\ref{TH3}.
\subsubsection{Zero delay margins}
It is important to stress that the fundamental limitations of, e.g. Theorem~\ref{TH2} would not apply to an actual plant in the strict sense. Models of the form~\eqref{Modele} are obviously simplistic and do not capture, e.g., the diffusivity that would stem from Kelvin-Voigt damping, or other phenomena that would be susceptible of making the delay margins non-null. However, these results do indicate
\begin{itemize}
\item that the delay-robustness margins would be very \emph{poor} for such systems
\item that controllers of the form~\eqref{Control_law_transport} or~\eqref{control} significantly trade off delay-robustness for performance, making them likely to be unusable. 
\end{itemize}
In this regard, a more quantitative approach to analyzing the performance--delay-robustness trade-offs made available by the use of backstepping is needed, in particular to assess whether the qualitative approach of the present article remains valid with more realistic models. A first step in analyzing this trade-offs has been taken in~\cite{Auriol2017}. In the next sections, we analyze the impact of the results on broader classes of systems.
\subsubsection{Interconnected systems}
An important focus point in the recent literature is the control of interconnected and cascade systems: Ordinary Differential Equations (ODEs) featuring hyperbolic systems in the actuation paths have, in particular, received a lot of attention~\cite{Bekiaris-Liberis2014,bresch2014adaptive,DiMeglio2017,sagert2013backstepping}. A recurrent motivation for studying such systems is the control of mechanical vibrations in drilling, where the hyperbolic PDEs correspond to axial and torsional waves traveling along the drillstring, while the ODE models the Bottom Hole Assembly (BHA) dynamics (see, e.g.,~\cite{bresch2016prediction,DiMeglio2015} for details). The strategy in most approaches consists in transforming the \emph{interconnected} systems into \emph{cascade} systems by canceling the reflection at the controlled boundary. This enables the design of predictor-like feedback laws, that focus on stabilizing the (potentially unstable) ODE. This approach, although rigorously correct, is bound to exhibit poor delay-robustness in practice, as detailed in the previous section. To illustrate this point, we consider the following example.
\begin{example}
Let us consider the following ODE-PDE system
\begin{align}
u_t(t,x)+u_x(t,x)&=0,\\
v_t(t,x)-v_x(t,x)&=0,
\end{align}
with the boundary conditions
\begin{align}
u(t,0)&=qv(t,0)+cx(t),\\
v(t,1)&=\rho u(t,1)+U(t),
\end{align}
where~$x$ satisfies
\begin{align}
\dot{x}=ax(t)+bv(t,0),
\end{align}
where~$a$,~$b$ and~$c$ are non-null constants. Considering~$u(t,1)$ as the output of the system and taking the Laplace transform, we get the following characteristic equation
\begin{align}
\boldsymbol{u}(s,1)=(q+\frac{b}{s-a})e^{-2s}(\rho \boldsymbol{u}(s,1)+\boldsymbol{U}(s)).
\end{align}
This yields
\begin{align}
&\dot{u}(t,1)-\rho q \dot{u}(t-2,1)-q \dot{U}(t-2)=\nonumber \\
&au(t,1)+(b-qa)(\rho u(t-2,1)+U(t-2)).
\end{align}
This equation has a similar structure to that of equation~\eqref{delay_transport3} in presence of a delay in the actuation. Some of the results described in this paper can then be adjusted for ODE-PDEs systems.
\end{example}
 Interestingly, this opens new perspectives for the control of PDE-ODE systems: when the reflection at the controlled boundary is partially or not canceled, the system takes in the general case the form of a neutral equation with distributed delay of the same order as the ODE, similarly to Equation~\eqref{delay_eq_alpha}. The stability of the target and closed-loop system is then subject to restrictive conditions on the coupling terms that have not been canceled by the controller, contrary to the idealized case where the closed-loop system is a cascade. In this regard, the stability analysis methods for such systems developed in~\cite{damak2015stability,niculescu2001delay,niculescu2001delays} will be instrumental.
\subsubsection{Systems with multiple equations}
Some physical systems require more than two equations to be properly modeled. For instance, Drift-Flux Models described in~\cite{Aarsnes2014} representing the flow of liquid and gas along the oil wells consist of three distributed equations of conservation. Along with closure relations, this yields a set of three nonlinear transport PDEs with appropriate boundary conditions. The model can be linearized around a given equilibrium profile, which yields a system of the form~\eqref{u_pde}--\eqref{IC} with~$n=2$ and~$m=1$  (see, e.g.,~\cite{aarsnes2016methodology}). To provide finite-time convergence, using a control law such as the one described in~\cite{di2013stabilization} or~\cite{auriol2016minimum} require to cancel all the reflexion terms. In light of what has been presented in this paper, this does not seem desirable in term of delay-robustness. Applying a transformation similar to the one proposed in section~\ref{transfo_back} and writing the corresponding characteristic equation would lead to a matrix neutral equation with distributed delays. One can then use classical method (\cite{hale2002strong}) to analyze such equations. However, the results presented here do not straightforwardly extend due to the presence of remaining terms in the target system. This will be the purpose of future contributions.
\section{Simulation results}\label{sec:simulation}
In this section we illustrate our results with simulations on a toy problem. The numerical values of the parameters are as follow.
\begin{align}
&\lambda=\mu=\sigma^{+-}=\sigma^{-+}=q=1, \quad \rho=0.85.
\end{align}
The (positive) delay in the loop is denoted~$\delta$. The parameters values are chosen such that 
\begin{itemize}
	\item the open-loop system is unstable (\cite{bastin2016stability})
	\item the open-loop gain satisfies~$\frac{1}{2}<|\rho q| <1$, so that the control law from~\cite{coron2013local} is not robustly stable to small delays (but the system can be delay-robustly stabilized).
\end{itemize}
\begin{figure}[htb]
		\includegraphics[width=\columnwidth]{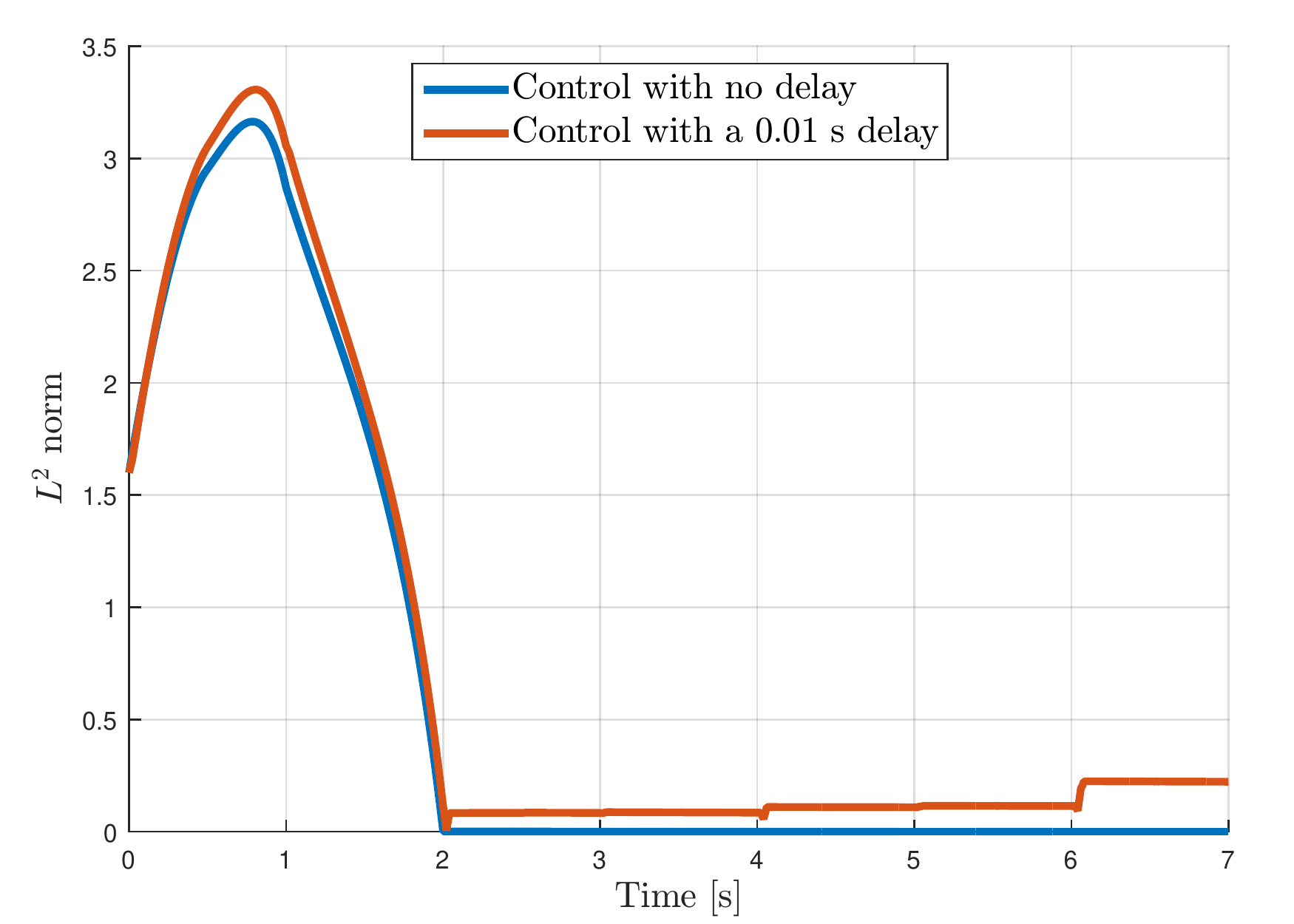}
	\caption{Time evolution of the L2-norm using the control law from~\cite{coron2013local} in presence of a delay}
	\label{fig:Controller Vazquez}
\end{figure}
Figure~$1$ pictures the~$\mathcal{L}^2-$norm of the state~$(u,v)$ using the control law presented in~\cite{coron2013local} without any delay ($\delta=0$~s) and then in presence of a small delay in the loop ($\delta=0.01$~s). As expected by the theory, with this control law, the system converges in finite time to its zero-equilibrium when there is no delay in the loop but becomes unstable in presence of a small delay.  
\begin{figure}[htb]
		\includegraphics[width=\columnwidth]{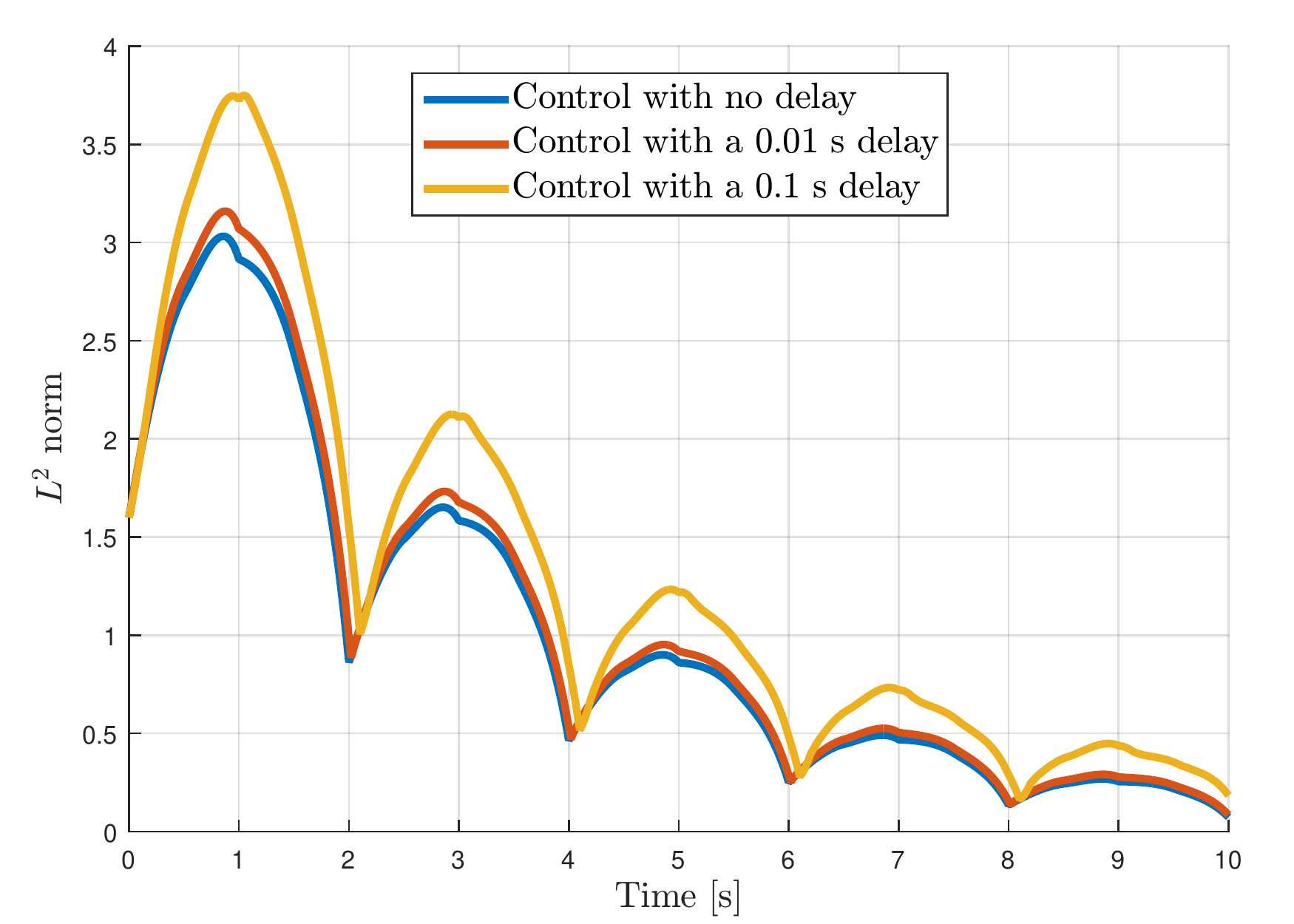}
	\caption{Time evolution of the L2-norm using the control law \eqref{control3} in presence of a delay}
	\label{fig:Controller paper}
\end{figure}
Figure~$3$ pictures the~$\mathcal{L}^2-$norm of the state~$(u,v)$ using the new control law \eqref{control3} ($K$ is chosen equal to~$0.1$) for the same situations ($\delta=0$ and~$\delta=0.01$~s) and for a larger delay ($\delta=0.1$~s). As expected by the theory, the system is now robustly stable to delays in the loop. However, this improvement in terms of delay margin comes at the cost of a diminution of the convergence rate. This example illustrates the importance of this change of paradigm.

\section{Concluding remarks}\label{sec:conclusion}
Inspired by the results obtained for a system of two transport equations, we have proved that control laws ensuring finite-time stabilization are often not robust to arbitrarily small delays in the loop due to the complete cancellation of the proximal reflection.  
Consequently it appeared essential to make a change of paradigm, focusing on delay-robust stabilization. This has been done by means of a tuning parameter enabling a trade-off between convergence rate and delay-robustness. Even if the delay-robustness properties of the output-feedback controller (crucial for any application on an industrial problem) were not studied there, the result presented in this paper is a new step towards a complete analysis of the properties of the backstepping controller. This change of paradigm implies to rethink the controller-design for PDEs-ODEs systems or systems with more than two PDEs. That will be the purpose of future contributions. 

\section*{Acknowledgment}
We thank Laurent Praly, Delphine Bresch-Pietri, Sebastien Boisgerault and Miroslav Krstic for their valuable comments.\\
The work of the second author was supported by the Research Council of Norway, ConocoPhillips, Det norske oljeselskap, Lundin, Statoil and Wintershall through the research center DrillWell (203525/O30) at IRIS.

\appendix

\section{Appendix}
In this appendix, we prove an important result of complex analysis. Let us consider some strictly positive integers~$n$ and~$m$, a sequence of constant matrices~$A_k$~$\in \mathcal{M}(n,n)$ and a sequence of positive constants~$\tau_k$. We consider the holomorphic function~$F$ defined for every complex number~$s$ by
\begin{align}
F(s)=I_n-\sum_{k=0}^m A_k e^{-s\tau_k},
\end{align}
where~$I_n$ is the identity matrix of dimension~$n$. For any real number~$\sigma>0$, we denote~$P_\sigma$ the open half-plane~$\{s \in \mathbb{C}~|~\Re(s) > \sigma \}$. 
We have the following general theorem
\begin{theorem}\label{Rouche_general}
Let us consider an holomorphic function~$H$ such that 
\begin{align}
|H(s)| \underset{|s|\to+\infty}{\longrightarrow} 0.
\end{align} 
If the function~$det(F)$ has an infinite number of zeros on~$P_\sigma$, then the function~$det(F+H)$ has an infinite number of zeros whose real parts are strictly positive.
\end{theorem} 
To prove this theorem, we slightly adjust the proof from~\cite{boisgerault2013growth}. For any positive~$\eta$, we denote~$Z_{\eta}$ the set of complex numbers whose distance to the zeros of det$(F)$ is at most~$\eta$:
\begin{align}
Z_{\eta}=\{z \in \mathbb{C}~|~\exists s \in \mathbb{C},~\det(F)(s)=0~\text{and}~|s-z|<\eta\}.
\end{align}
We start proving the following lemma
\begin{lemma}\label{Lemma_key}\textbf{-zero clusters and lower bound.} Let us consider~$\sigma>0$ and~$\epsilon>0$. There exists~$\eta>0$ such that any connected component~$\Lambda$ of the set~$Z_\eta$ is bounded and such that~$\Lambda \subset P_{\sigma-\epsilon}$ if~$\Lambda \cap P_\sigma \ne \emptyset$. Moreover, there exists~$\kappa>0$ such that~$|det(F)|\geq \kappa$ on~$P_{\sigma-\epsilon}\backslash Z_\eta$ 
\end{lemma}
\begin{proof}\em
This proof is similar to the one given in~\cite{boisgerault2013growth}. By continuity of the determinant we have
\begin{align}
\underset{\Re(s) \to+\infty}{\lim} \det(F(s))=1, \label{eq_1_pr}
\end{align}
Let us denote~$N(\rho)$ the number of zeros of~$det(F)$ whose modulus is smaller than~$\rho$. The function~$s \mapsto \det F(is)$ satisfies the assumptions of~\cite[Theorem VIII]{levinson1940gap}. Thus, there exists a positive constant~$K$ such that
\begin{align}
 \underset{\rho \to+\infty}{\limsup} ~\frac{N(\rho)}{\rho} \leq \frac{K}{\pi}.
\end{align}
If there is an unbounded connected component of~$Z_\eta$, then there exists a sequence ($z_p$) of distinct zeros of~$\det(F)$ such that for any~$p \in \mathbb{N}$,~$|z_{p+1}-z_p| \leq 2\eta$. It yields
\begin{align}
\underset{\rho \to+\infty}{\limsup} ~\frac{N(\rho)}{\rho} \geq \frac{1}{2\eta}.
\end{align}
Choosing~$\eta < \frac{\pi}{2k}$ implies that any connected component of the set~$Z_\eta$ is bounded.

Consider now the following complex analysis result. Let us consider a sequence~$s_p$ of numbers in~$P_{\sigma-\epsilon}$ such that 
\begin{align}
\Re s_p \underset{p \to+\infty}{\longrightarrow} x \in \mathbb{R}.
\end{align}
The sequence~$f_p(s)=\det(F(s+i\Im s_p))$ is locally bounded on~$\mathbb{C}$ uniformly in $p$. Consequently, there exists a sub-sequence that converges locally uniformly to an entire function~$f_{\text{lim}}$ (Montel's theorem). Due to \eqref{eq_1_pr}, this function cannot be identically zero. If~$f_{\text{lim}}=0$, we define~$m$ as the multiplicity of~$x$. Otherwise we set~$m=0$. Using Hurwitz's theorem, we get that for~$\alpha>0$ small enough, and for~$p$ large enough,~$f_p$ has precisely~$m$ zeroes in the open disk~$B(x,\alpha)$. Since~$f_p(x)=\det(F(x+ i\Im s_p))=\det(F(s_p))$, it yields that~$\det(F)~\text{has~$m$ zeros in B($s_p,\alpha)$}$.

Let us consider a sequence~$\Lambda_p$ of bounded connected components of~$Z_{\frac{1}{2p}}$, defined for~$p$ large enough such that~$\Lambda_p \cap P_\sigma \ne \emptyset$ and~$\Lambda_p \backslash P_{\sigma-\epsilon} \ne \emptyset$. For any such~$p$ and for~$0<\alpha < \epsilon$ small enough, we can find a real~$y_p$ such that~$\det(F)$ has a least~$p\alpha-1$ zeros in the open disk~$B(\sigma+iy_n,\alpha)$. This contradicts the result of the previous paragraph for the sequence~$s_p=\sigma+iy_p$. 

By the contradiction, let us assume that~$|det(F)|$ has no strictly positive lower bound on~$P_{\sigma-\epsilon}\backslash Z_\eta$. We can then find a sequence~$s_p$ such that
\begin{align}
 \det(F)(s_p)  \underset{p \to+\infty}{\longrightarrow}0.
\end{align}
Because of \eqref{eq_1_pr}, this sequence can be selected such that~$\Re s_p$ converges to~$x \in \mathbb{R}$. It yields
\begin{align}
f_{\lim}(x)&=\underset{p \to+\infty}{\lim} f_p(x)=\underset{p \to+\infty}{\lim} f_p(s_p-\Im s_p)\\
&=\underset{p \to+\infty}{\lim} \det(F(s_p))=0.
\end{align}
Consequently, there is an integer~$p_0$ such that~$\det(F)$ has at least one zero in~$B(s_{p_0},\eta)$. So,~$s_{p_0} \in Z_{\eta}$ which is a contradiction.
\end{proof}
We now prove Theorem~\ref{Rouche_general}
\begin{proof}\em
Let us consider~$0<\epsilon<\sigma$,~$\det(F)$ has an infinite number of zeros on~$P_{\sigma}$. Let~$\eta>0$ be such that any connected component~$\Lambda$ of~$Z_\eta$ that contains such a zero is bounded and included in~$P_\epsilon$. Since the zeros of~$\det(F)$ are isolated, every~$\Lambda$ contains a finite number of zeros, and the collection of sets~$\Lambda$ is infinite.

Let us consider~$\Lambda_k$ a sequence of connected component of~$Z_\eta$. Since these components are bounded, we can define~$\Gamma_k$ as the closed contour of~$\Lambda_k$. Due to Lemma-\ref{Lemma_key}, there exists~$M>0$ such that 
\begin{align}
\forall k \in \mathbb{N},~\forall s \in \Gamma_k,~|\det(F(s))|>M.
\end{align} 
By assumption~$H$ is such that
\begin{align}
|H(s)| \underset{|s|\to+\infty}{\longrightarrow} 0.
\end{align}
Since~$F$ is upper-bounded, it yields (developing the determinant) 
\begin{align}
|\det((F+H)(s))-\det(F(s))| \underset{|s|\to+\infty}{\longrightarrow} 0.
\end{align}
Consequently, for~$k$ large enough, we get
\begin{align}
\forall s \in \Gamma_k,~|\det((F+H)(s))-\det(F(s))|<M.
\end{align} 
Rouch$\acute{e}$'s theorem implies that~$\det(F)$ and~$\det(F+H)$ have the same number of zeros inside of each~$\Gamma_k$ ($k>k_0$). Consequently~$\det(F+H)$ has an infinite number of zeros on~$P_\epsilon$. So~$\det(F+H)$ has an infinite number of zeros whose real part is strictly positive. This concludes the proof.
\end{proof}

\bibliographystyle{plain}  
\bibliography{Biblio}

\begin{thebibliography}{10}

\bibitem{Aarsnes2014}
U.~J.~F. Aarsnes, F.~Di~Meglio, S.~Evje, and O.-M. Aamo.
\newblock Control-oriented drift-flux modeling of single and two-phase flow for
  drilling.
\newblock In {\em Proceedings of the ASME 2014 Dynamic Systems and Control
  Conference, San Antonio, USA, October 22-24}, 2014.

\bibitem{aarsnes2013limits}
Ulf Jakob~F Aarsnes, Ole~Morten Aamo, Espen Hauge, and Alexey Pavlov.
\newblock Limits of controller performance in the heave disturbance attenuation
  problem.
\newblock In {\em Control Conference (ECC), 2013 European}, pages 1071--1076.
  IEEE, 2013.

\bibitem{aarsnes2016methodology}
Ulf Jakob~F Aarsnes, Florent Di~Meglio, Robert Graham, Ole~Morten Aamo, et~al.
\newblock A methodology for classifying operating regimes in
  underbalanced-drilling operations.
\newblock {\em SPE Journal}, 21(02):423--433, 2016.

\bibitem{aarsnes2014modeling}
Ulf Jakob~Fl{\o} Aarsnes, Martin~Standal Gleditsch, Ole~Morten Aamo, Alexey
  Pavlov, et~al.
\newblock Modeling and avoidance of heave-induced resonances in offshore
  drilling.
\newblock {\em SPE Drilling \& Completion}, 29(04):454--464, 2014.

\bibitem{anfinsen2016adaptive}
Henrik Anfinsen, Mamadou Diagne, Ole~Morten Aamo, and Miroslav Krstic.
\newblock An adaptive observer design for $ n+ 1$ coupled linear hyperbolic
  pdes based on swapping.
\newblock {\em IEEE Transactions on Automatic Control}, 61(12):3979--3990,
  2016.

\bibitem{anfinsen2016boundary}
Henrik Anfinsen, Mamadou Diagne, Ole~Morten Aamo, and Miroslav Krstic.
\newblock Boundary parameter and state estimation in general linear hyperbolic
  pdes.
\newblock {\em IFAC-PapersOnLine}, 49(8):104--110, 2016.

\bibitem{Auriol2017}
J.~Auriol, U.~J. Aarsnes, and F.~Di~Meglio.
\newblock Performance trade-offs in the observer design of a $2\times 2$ linear
  hyperbolic system.
\newblock {\em submitted to IEEE Conference on Decision and Control}, 2017.

\bibitem{auriol2017performance}
Jean Auriol, Ulf Jakob~F Aarsnes, and Florent Di~Meglio.
\newblock Performance trade-offs in the observer design of a 2$\times$ 2 linear
  hyperbolic system.
\newblock {\em Preprints - American Chemical Society, Division of Petroleum
  Chemistry}, 2017.

\bibitem{auriol2016minimum}
Jean Auriol and Florent Di~Meglio.
\newblock Minimum time control of heterodirectional linear coupled hyperbolic
  pdes.
\newblock {\em Automatica}, 71:300--307, 2016.

\bibitem{bastin2016stability}
Georges Bastin and Jean-Michel Coron.
\newblock {\em Stability and boundary stabilization of 1-d hyperbolic systems}.
\newblock Springer, 2016.

\bibitem{Bekiaris-Liberis2014}
N.~Bekiaris-Liberis and M.~Krstic.
\newblock Compensation of wave actuator dynamics for nonlinear systems.
\newblock {\em IEEE Transactions on Automatic Control}, 59(6):1555--1570, June
  2014.

\bibitem{boisgerault2013growth}
S{\'e}bastien Boisg{\'e}rault.
\newblock Growth bound of delay-differential algebraic equations.
\newblock {\em Comptes Rendus Mathematique}, 351(15):645--648, 2013.

\bibitem{bresch2016prediction}
Delphine Bresch-Pietri and Florent Di~Meglio.
\newblock Prediction-based control of linear input-delay system subject to
  state-dependent state delay-application to suppression of mechanical
  vibrations in drilling.
\newblock {\em IFAC-PapersOnLine}, 49(8):111--117, 2016.

\bibitem{bresch2010delay}
Delphine Bresch-Pietri and Miroslav Krstic.
\newblock Delay-adaptive predictor feedback for systems with unknown long
  actuator delay $ $.
\newblock {\em IEEE Transactions on Automatic Control}, 55(9):2106--2112, 2010.

\bibitem{bresch2014adaptive}
Delphine Bresch-Pietri and Miroslav Krstic.
\newblock Adaptive output feedback for oil drilling stick-slip instability
  modeled by wave pde with anti-damped dynamic boundary.
\newblock In {\em American Control Conference (ACC), 2014}, pages 386--391.
  IEEE, 2014.

\bibitem{coron1999lyapunov}
Jean-Michel Coron, Brigitte Andr{\'e}a-Novel, and Georges Bastin.
\newblock A lyapunov approach to control irrigation canals modeled by
  saint-venant equations.
\newblock In {\em Proc. European Control Conference, Karlsruhe}, 1999.

\bibitem{coron2017finite}
Jean-Michel Coron, Long Hu, and Guillaume Olive.
\newblock Finite-time boundary stabilization of general linear hyperbolic
  balance laws via fredholm backstepping transformation.
\newblock {\em arXiv preprint arXiv:1701.05067}, 2017.

\bibitem{coron2013local}
Jean-Michel Coron, Rafael Vazquez, Miroslav Krstic, and Georges Bastin.
\newblock Local exponential $h^2$ stabilization of a 2$\times$ 2 quasilinear
  hyperbolic system using backstepping.
\newblock {\em SIAM Journal on Control and Optimization}, 51(3):2005--2035,
  2013.

\bibitem{curtain2009transfer}
Ruth Curtain and Kirsten Morris.
\newblock Transfer functions of distributed parameter systems: A tutorial.
\newblock {\em Automatica}, 45(5):1101--1116, 2009.

\bibitem{damak2015stability}
S{\'e}rine Damak, Michael Di~Loreto, and Sabine Mondi{\'e}.
\newblock Stability of linear continuous-time difference equations with
  distributed delay: Constructive exponential estimates.
\newblock {\em International Journal of Robust and Nonlinear Control},
  25(17):3195--3209, 2015.

\bibitem{datko1986example}
R~Datko, J~Lagnese, and MP~Polis.
\newblock An example on the effect of time delays in boundary feedback
  stabilization of wave equations.
\newblock {\em SIAM Journal on Control and Optimization}, 24(1):152--156, 1986.

\bibitem{de2003boundary}
Jonathan de~Halleux, Christophe Prieur, J-M Coron, Brigitte d'Andr{\'e}a Novel,
  and Georges Bastin.
\newblock Boundary feedback control in networks of open channels.
\newblock {\em Automatica}, 39(8):1365--1376, 2003.

\bibitem{DiMeglio2015}
F.~Di~Meglio and U.~J.~F. Aarsnes.
\newblock A distributed parameter systems view of control problems in drilling.
\newblock In {\em to appear in 2nd IFAC Workshop on Automatic Control in
  Offshore Oil and Gas Production, Florian\'opolis, Brazil}, 2015.

\bibitem{DiMeglio2017}
F.~Di~Meglio, F.~Bribiesca~Argomedo, L.~Hu, and M.~Krstic.
\newblock Stabilization of coupled linear heterodirectional hyperbolic pde-ode
  systems.
\newblock {\em accepted for publication in Automatica}, 2017.

\bibitem{di2011dynamics}
Florent Di~Meglio.
\newblock {\em Dynamics and control of slugging in oil production}.
\newblock PhD thesis, {\'E}cole Nationale Sup{\'e}rieure des Mines de Paris,
  2011.

\bibitem{di2013stabilization}
Florent Di~Meglio, Rafael Vazquez, and Miroslav Krstic.
\newblock Stabilization of a system of $ n+ 1$ coupled first-order hyperbolic
  linear pdes with a single boundary input.
\newblock {\em IEEE Transactions on Automatic Control}, 58(12):3097--3111,
  2013.

\bibitem{dudret2012stability}
St{\'e}phane Dudret, Karine Beauchard, Fouad Ammouri, and Pierre Rouchon.
\newblock Stability and asymptotic observers of binary distillation processes
  described by nonlinear convection/diffusion models.
\newblock In {\em American Control Conference (ACC), 2012}, pages 3352--3358.
  IEEE, 2012.

\bibitem{egeland2002modeling}
Olav Egeland and Jan~Tommy Gravdahl.
\newblock {\em Modeling and simulation for automatic control}, volume~76.
\newblock Marine Cybernetics Trondheim, Norway, 2002.

\bibitem{fabiano2013stability}
Richard~H Fabiano.
\newblock A stability result for a scalar neutral equation with multiple
  delays.
\newblock In {\em Decision and Control (CDC), 2013 IEEE 52nd Annual Conference
  on}, pages 1089--1094. IEEE, 2013.

\bibitem{hale2002strong}
Jack~K Hale and Sjoerd M~Verduyn Lunel.
\newblock Strong stabilization of neutral functional differential equations.
\newblock {\em IMA Journal of Mathematical Control and Information}, 19(1 and
  2):5--23, 2002.

\bibitem{hu2016control}
Long Hu, Florent Di~Meglio, Rafael Vazquez, and Miroslav Krstic.
\newblock Control of homodirectional and general heterodirectional linear
  coupled hyperbolic pdes.
\newblock {\em IEEE Transactions on Automatic Control}, 61(11):3301--3314,
  2016.

\bibitem{karafyllis2014relarion}
Iasson Karafyllis and Miroslav Krtic.
\newblock On the relation of delay equations to first-order hyperbolic partial
  differential equations.
\newblock {\em ESAIM: Control, Optimisation and Calculus of Variations},
  20(3):894--923, 2014.

\bibitem{krstic2008lyapunov}
Miroslav Krstic.
\newblock Lyapunov tools for predictor feedbacks for delay systems: Inverse
  optimality and robustness to delay mismatch.
\newblock {\em Automatica}, 44(11):2930--2935, 2008.

\bibitem{kyllingstad2009new}
Age Kyllingstad, Pal~Jacob Nessj{\o}en, et~al.
\newblock A new stick-slip prevention system.
\newblock In {\em SPE/IADC Drilling Conference and Exhibition}. Society of
  Petroleum Engineers, 2009.

\bibitem{levinson1940gap}
Norman Levinson.
\newblock {\em Gap and density theorems}, volume~26.
\newblock American mathematical society New York, 1940.

\bibitem{li2010strong}
Tatsien Li and Bopeng Rao.
\newblock Strong (weak) exact controllability and strong (weak) exact
  observability for quasilinear hyperbolic systems.
\newblock {\em Chinese Annals of Mathematics, Series B}, 31(5):723--742, 2010.

\bibitem{logemann1996conditions}
Hartmut Logemann, Richard Rebarber, and George Weiss.
\newblock Conditions for robustness and nonrobustness of the stability of
  feedback systems with respect to small delays in the feedback loop.
\newblock {\em SIAM Journal on Control and Optimization}, 34(2):572--600, 1996.

\bibitem{niculescu2001delay}
Silviu-Iulian Niculescu.
\newblock {\em Delay effects on stability: a robust control approach}, volume
  269.
\newblock Springer Science \& Business Media, 2001.

\bibitem{niculescu2001delays}
Silviu-Iulian Niculescu.
\newblock On delay-dependent stability under model transformations of some
  neutral linear systems.
\newblock {\em International Journal of Control}, 74(6):609--617, 2001.

\bibitem{sagert2013backstepping}
Conrad Sagert, Florent Di~Meglio, Miroslav Krstic, and Pierre Rouchon.
\newblock Backstepping and flatness approaches for stabilization of the
  stick-slip phenomenon for drilling.
\newblock {\em IFAC Proceedings Volumes}, 46(2):779--784, 2013.

\bibitem{thorp1998electromechanical}
James~S Thorp, Charles~E Seyler, and Arun~G Phadke.
\newblock Electromechanical wave propagation in large electric power systems.
\newblock {\em IEEE Transactions on Circuits and Systems I: Fundamental Theory
  and Applications}, 45(6):614--622, 1998.

\bibitem{xu2002exponential}
Cheng-Zhong Xu and Gauthier Sallet.
\newblock Exponential stability and transfer functions of processes governed by
  symmetric hyperbolic systems.
\newblock {\em ESAIM: Control, Optimisation and Calculus of Variations},
  7:421--442, 2002.

\end{thebibliography}

\begin{IEEEbiographynophoto}{Jean Auriol}
Jean Auriol is a PhD student at Centre Automatique et Systèmes of MINES ParisTech, part of PSL Research University, under the direction of Florent Di Meglio.  His current research interests include control and estimation design for hyperbolic PDEs 
\end{IEEEbiographynophoto}

\begin{IEEEbiographynophoto}{Ulf Jakob F. Aarsnes}
Ulf Jakob F. Aarsnes works as a PostDoc at the International Research Institute of Stavanger (IRIS) in the DrillWell Project. He graduated from the Norwegian University of Science and Technology with a MSc in 2012 and defended his PhD in 2016, both in the field of Engineering Cybernetics. His research interests include modelling, analysis and control of distributed parameter systems and their industrial applications. Aarsnes has authored and/or coauthored more than 20 technical papers in this field.
\end{IEEEbiographynophoto}

\begin{IEEEbiographynophoto}{Philippe Martin}
	is a Senior Researcher at MINES ParisTech, PSL Research University, France. He received in 1992 the PhD degree in 
	Mathematics and Control from Ecole des Mines de Paris (the former name of MINES ParisTech). His interests include 
	theoretical aspects of control theory (in particular nonlinear control, observers, and partial differential equations),
	as well as industrial control problems in several fields of engineering (in particular electrical machines, mechatronics,
	and aerial vehicles).
\end{IEEEbiographynophoto}

\begin{IEEEbiographynophoto}{Florent Di Meglio}
Florent Di Meglio is tenured professor at the Centre Automatique et Systèmes of MINES ParisTech, part of PSL Research University. He received his Ph.D. from the same university in Mathematics and Control in 2011, and was a Posdoctoral Researcher at UC San Diego from 2011 to 2012. His current research interests include control and estimation design for hyperbolic PDEs, with application to process control, most notably multiphase flow control and oil drilling. 
\end{IEEEbiographynophoto}

\end{document}